\documentclass{article}
\usepackage[utf8]{inputenc}
\usepackage{pb-diagram}
\usepackage{amsmath, amscd}
\usepackage{tikz-cd}
\usepackage{amsthm}
\usepackage{mathtools}
\usepackage{amsfonts}
\usepackage{amssymb}
\usepackage{graphicx}
\usepackage{xcolor}
\usepackage{hyperref}
\usepackage{soul}
\usepackage{imakeidx}

\newtheorem{theorem}{Theorem}[section]
\newtheorem{definition}{Definition}[section]

\newtheorem{lemma}[theorem]{Lemma}
\newtheorem{cor}[theorem]{Corollary}

\newtheorem{prop}[theorem]{Proposition}

\theoremstyle{remark}
\newtheorem{remark}[theorem]{Remark}

\newcommand{\be}{\begin{enumerate}}
\newcommand{\ee}{\end{enumerate}}
\newcommand{\beq}{\begin{equation}}
\newcommand{\eeq}{\end{equation}}

\newcommand{\Th}{\textrm{Th}}

\newcommand{\Lrings}{\mathcal{L}_{\text{rings}}}

\def\N{{\mathbb{N}}}
\def\Z{{\mathbb{Z}}}

\def\Q{{\mathbb{Q}}}
\def\Spec{\text{Spec}}

\def\MA{{\mathbb{A}}}
\def\MB{{\mathbb{B}}}

\def\D{{\mathcal{D}}}

\def\barA{\hat{A}}
\def\barB{\hat{B}}
\def\bara{\hat{a}}

\def\barx{\hat{x}}
\def\bary{\hat{y}}

\def\Ext{\text{Ext}}

{}
{}
{}

\title{Arbitrary models of the complete first-order theories of FDZ-rings }

\author{Mahmood Sohrabi \footnote{Stevens Institute of Technology.}} 
\date{}

\makeindex
\begin{document}

\maketitle

\begin{abstract}
In this paper, we study arbitrary models of the first-order theory of a  ring $A$ where the additive group $A$ is a finitely generated abelian group. Following an earlier paper by this author, Alexei G. Myasnikov and Francis Oger~\cite{MOS}, we call these rings the \emph{FDZ-rings} or \emph{FDZ-algebras}. The rings considered are not necessarily unitary, commutative, or associative. We provide criteria for such rings to be quasi finitely axiomatizable (QFA) or bi-interpretable with the ring of integers $\Z$. We shall also describe all rings elementarily equivalent to such a ring $A$ given certain constraints on $A$.
\end{abstract}

\tableofcontents

\maketitle

\section{Introduction}
In this paper, we extend the results of \cite{MOS} and continue the program of developing a structure theory for arbitrary models of finite dimensional rings over $\mathbb{Z}$ (FDZ-rings). Our motivation, at least partly, comes from the model theory of finitely generated nilpotent groups: via the various types of ``Mal'cev-like" correspondences between rings and nilpotent groups, understanding elementary equivalence classes on the ring side provides a route toward describing all groups elementarily equivalent to a given finitely generated nilpotent group. 

After recalling the relevant notions of first-order rigidity and quasi-finite axiomatizability, we focus on bi-interpretability with the ring of integers. We provide conditions that guaranty an FDZ-ring cannot be bi-interpretable with $\mathbb{Z}$, and we also offer sufficient conditions to ensure bi-interpretability. For the class of super tame FDZ-rings, we provide a complete description of arbitrary models of their complete first-order theory. Finally, we attempt such a characterization for all FDZ-rings whose ``maximal" scalar ring is bi-interpretable with the ring of integers $\Z$. These models will be called \emph{abelian deformations} of a tensor completion of $A$ over a ring elementarily equivalent to $\Z$. We describe these using certain symmetric 2-cocycles.

\subsection{First-order rigidity and quasi-finite axiomatizability}
We begin by fixing some notation and recalling some model-theoretic definitions (we refer to \cite{hodges} for details). For the remainder of this section, $\mathcal{L}$ is a countable language and $T$ is a complete $\mathcal{L}$-theory with infinite models.
\begin{definition}
A finitely generated $\mathcal{L}$-structure $\MA$ is called \emph{first-order rigid} if, for any finitely generated $\mathcal{L}$-structure $\MB$, first-order equivalence $\MA \equiv \MB$ implies isomorphism $\MA \simeq \MB$.
\end{definition}

\begin{definition}	Fix a finite signature. An infinite finitely generated  structure is Quasi Finitely Axiomatizable (QFA)\index{QFA model} if there exists a first-order sentence $\phi$ of the signature such that 
	\begin{itemize}
		\item $\MA\models \phi$
		\item If $\MB$ is a finitely generated structure in the same signature and $\MB\models \phi$ then $\MA\cong \MB$.
\end{itemize} \end{definition}

\begin{definition}
Suppose that structures $\mathbb A, \mathbb B$ in finite signatures are given, as well as
interpretations of $\mathbb A$ in~$\mathbb B$, and vise versa. Then an isomorphic copy~$\widetilde{\mathbb A}$ of~$\mathbb A$ can be defined in~$\mathbb A$ by ``decoding''~$\mathbb A$ from the copy of~$\mathbb B$ defined in~$\mathbb A$.
Similarly, an isomorphic copy $\widetilde{\mathbb B}$ of~$\mathbb B$ can be defined in~$\mathbb B$. An isomorphism $\Phi: {\mathbb A}\cong \widetilde{\mathbb A}$ can be viewed as a relation on~$\mathbb A$, and similarly for an isomorphism $\mathbb B \cong \widetilde{\mathbb B}$.
We say that $\mathbb A$ and $\mathbb B$ are \emph{bi-interpretable} (with parameters) if there
 exist isomorphisms that are first-order definable. \end{definition}
 
See (\cite{Nies2007}, Sec. 5) or \cite{Rich} for further details. For a careful and detailed study of various notions (weak and strong) of bi-interpretability, we refer the reader to~\cite{AKNS}. Here, we are most interested in bi-interpretability with the ring of integers $\Z$, where the above notions coincide.
\begin{theorem}(\cite{Nies2007}) A finitely generated structure $\mathbb A$ that is bi-interpretable with the ring of integers $\Z$ is QFA.\end{theorem}
\subsection{Preliminaries on FDZ-rings}

We call a \emph{ring} any structure $(A,+,\cdot)$\ such that $(A,+)$ is an
abelian group and $(x,y)\mapsto x\cdot y$\ is bilinear. A ring is said to be a 
\emph{scalar ring} if it is commutative, associative, and unitary. We denote
by $%
\mathbb{N}
^{\ast }$\ the set of strictly positive integers.

The reader is referred to [2] for the concepts of mathematical logic that
are used here. For each language $\mathcal{L}$, two $\mathcal{L}$-structures 
$\mathbb A$ and  $\mathbb B$ are said to be \emph{elementarily equivalent} if they satisfy the same
first-order sentences.

Any language $\mathcal{L}$ consists of relation symbols, constant symbols
, and function symbols. The $\mathcal{L}$-structures can be interpreted in a
language without function symbols, replacing each function with the relation
associated with its graph. The notions of elementary equivalence for the two
languages coincide.

For a pair of structures, for example, a scalar ring $R$ and a $R$-module $M$, it
is convenient to use a language without function symbols and to interpret each
operation, for example, the action of $R$ on $M$, as a relation on the
disjoint union of the two structures. We obtain the same notion of
elementary equivalence if we interpret each pair as a multi-based model,
possibly with function symbols (see [3]).

According to the definition given by Andr\'{e} Nies in [6], for each finite
language $\mathcal{L}$ with function symbols, a finitely generated $\mathcal{%
L}$-structure is said to be \emph{quasi-finitely axiomatizable} if it is
characterized among finitely generated $\mathcal{L}$-structures by one
sentence.

For each language $\mathcal{L}$ containing $+$ and each $\mathcal{L}$
-structure $M$ such that $(M,+)$ is an abelian group, we say that $M$ is 
\emph{finite dimensional over }$%
\mathbb{Z}
$ or \emph{FDZ}\ if $(M,+)$\ is finitely generated. We say that $M$ is \emph{%
FDZ-finitely axiomatizable} if it is characterized among FDZ $\mathcal{L}$%
-structures by one sentence.

We note that for each structure $M$ and each finite sequence $\overline{x}$
that generates $M$ (resp. $(M,+)$), the structure $(M,\overline{x})$ is
quasi-finitely (resp. FDZ-finitely) axiomatizable if and only if there
exists a formula $\varphi (\overline{u})$ satisfied by $\overline{x}$ in $M$
such that for each finitely generated (resp. FDZ) structure $N$ and each $%
\overline{y}$\ which satisfies $\varphi $ in $N$, the map $\overline{x}%
\rightarrow \overline{y}$ induces an isomorphism from $M$ to $N$.\bigskip

In~\cite{MOS} the following statements have been proved:

\begin{theorem}\label{MOS-1} Consider an FDZ scalar ring $R$ and a finite
sequence $\overline{r}$\ which generates $(R,+)$. Then $(R,\overline{r})$ is
FDZ-finitely axiomatizable.\end{theorem}

\begin{theorem} \label{MOS-2} Consider an FDZ scalar ring $R$, an FDZ $R$%
-module $M$ and some finite sequences $\overline{r}$\ and $\overline{x}$\
which generate $(R,+)$ and $(M,+)$. Then $(R,\overline{r},M,\overline{x})$
is FDZ-finitely axiomatizable.\end{theorem}

\begin{theorem} \label{MOS-3} Consider an FDZ scalar ring $R$, an FDZ ring 
$A$ equipped with a structure of $R$-module such that the multiplication of $%
A$ is $R$-bilinear, and some finite sequences $\overline{r}$\ and $\overline{%
x}$\ that generate $(R,+)$ and $(A,+)$. Then $(R,\overline{r},A,\overline{x}%
)$ is FDZ-finitely axiomatizable.\end{theorem}

For each ring $(A,+,.)$, we consider the two-sided ideals:
\begin{itemize}
\item $Ann(A)=\{x\in R\mid xy=yx=0$ for each $y\in A\}$,

\item $A^{2}=\{x_{1}y_{1}+\cdots +x_{n}y_{n}\mid n\in 
\mathbb{N}
$ and $x_{1},y_{1},\ldots ,x_{n},y_{n}\in A\}$,
\item $Is(I)=\{x\in A\mid nx\in I$ for some $n\in 
\mathbb{N}
^{\ast }\}$ for each two-sided ideal $I$ of $A$,
\item $\Delta(A)=Is(A^2)$
\item$K(A)=Ann(A)+\Delta(A)$ 
\item $L(A)=Is(Ann(A)+A^{2})$
\item $M(A)=A/L(A)$ and 
\item $N(A)=L(A)/K(A)$.
\end{itemize}
We say that $A$\ is \emph{regular} if $K(A)=L(A)$. Then, for each additive subgroup $A_0\subset
A $ such that $Ann(A)=A_0\oplus (Ann(A)\cap \Delta(A))$, there exists a subring $A_f
\subset A $ containing $\Delta(A)$ such that $A=A_f\oplus A_0$. Such a subgroup $A_0$ is called an \emph{addition} of $A$, and $A_f$ is called the corresponding \emph{foundation} of $A$. We say that $A$\ is \emph{tame} if $Ann(A)\subset \Delta(A)$.

If $A$ is not regular and an addition $A_0$ exists, the corresponding $A_f$ is not necessarily a subring. In that case, we refer to the quotient $A_f=A/A_0$ as the foundation $A_f$ of $A$ corresponding to $A_0$.

The following result is a generalization of Theorem~\ref{MOS-1}:

\begin{theorem}\label{MOS-4} Consider a tame FDZ ring $A$ and a finite
sequence $\overline{x}$\ that generates $(A,+)$. Then $(A,\overline{x})$ is
FDZ-finitely axiomatizable.\end{theorem}

In contrast to the results above, we have the following:

\begin{theorem} For any FDZ rings $A$ and $B$ and each $k\in 
\mathbb{N}\setminus \{0\}
$ which is divisible by $\left| L(B)/K(B)\right| $, the following
properties are equivalent:
\be 
\item $A$ and $B$ are elementarily equivalent;

\item there exists an injective homomorphism $f:A\rightarrow B$ that
induces some isomorphisms from $Is(A^{2})$ to $Is(B^{2})$ and from $%
A/Ann(A)\ $to $B/Ann(B)$, such that $\left| B:f(A)\right| $ is finite
and prime to $k$;

\item There exists an isomorphism $%
\mathbb{Z}
_{0}\times A\cong 
\mathbb{Z}
_{0}\times B$, where $%
\mathbb{Z}
_{0}=(%
\mathbb{Z}
,+,\ast )$\ with $x\ast y=0$\ for any $x,y\in 
\mathbb{Z}
$.\ee
\end{theorem}
\begin{cor}\label{MOS-6} Each regular FDZ ring is first-order rigid among the FDZ rings. \end{cor}

In ~\cite{Rich}, the following theorem was proved.

\begin{theorem}\label{KMS-1}  The following are equivalent for a finitely generated commutative ring $A$ (without unity):
\begin{enumerate}
    \item $A$ is $QFA$ in the language of non-unitary rings.
    \item $A$ is tame.
\end{enumerate}
\end{theorem}
\subsection{QFA and bi-interpretability with \texorpdfstring{$\Z$}{Z} for scalar rings}
In the sequel, we shall refer to two crucial results obtained by M. Aschenbrenner et. al in \cite{AKNS}.

\begin{theorem}[\cite{AKNS}]\label{AKNS1} Suppose that a commutative, associative, unitary ring $A$ is finitely generated, and let $N$ be the nilradical of $A$. Then $A$ is bi-interpretable with the ring of integer $\Z$  if and only if $A$ is infinite, $\text{Spec}^0(A)$ is connected, and there is some integer $d\geq 1$ with $dN=0$.   
\end{theorem}
\begin{theorem}[\cite{AKNS}]\label{AKNS2} Each finitely generated associative, commutative, unitary ring is QFA.
\end{theorem}
\subsection{Main results}
Here, we state the main results of this paper. The first result is a generalization of
\begin{theorem}\label{thm:Main1} The following are equivalent for an $FDZ$-ring.
\begin{enumerate}
    \item $A$ is $QFA$ in the language of rings.
    \item $A$ is tame.
\end{enumerate}
\end{theorem}
\begin{theorem}\label{thm:main2}  Let $A$ be an FDZ-ring. If the annihilator $Ann(A)$ is infinite and $\Delta(A)\neq A$ then $A$ is not bi-interpretable with the ring of integers $\Z$.\end{theorem}
For the next definition, we need the notion of the largest scalar ring $P(f_A)$ of $A$. We refer the reader to Section~\ref{sec:bilin} for a definition.
\begin{definition}[Super tame FDZ-rings] Let $A$ be an $FDZ-ring$. We say that $A$ is super tame if the following hold:
\begin{enumerate}
\item The punctured spectrum $\text{Spec}^\circ(P(f_A))$ of $P(f_A)$ is connected.
\item Either $Ann(A)$ is finite or $A=\Delta(A)$.
\end{enumerate}
 \end{definition}

\begin{theorem}\label{thm:Main3} Let $A$ be an FDZ-ring. If $A$ is super tame, then $A$ is bi-interpretable with $\Z$.\end{theorem}

The proofs of the above theorems are included in Sections~\ref{sec:bilin} and \ref{sec:biint}.

In the later sections, we shall attempt to characterize all rings elementarily equivalent to an FDZ-ring $A$, given some restrictions on $A$. The hope is to provide a complete characterization in an upcoming work.

The proof of the following theorem is discussed in Section~\ref{sec:A-super-tame}.
\begin{theorem}\label{thm:main4} Let $A$ be a super tame FDZ-ring and assume that $B$ is a ring. Then the following are equivalent:
\begin{enumerate}
    \item $A\equiv B$
    \item $B\cong R\otimes_{\Z}A$ for some scalar ring $R\equiv \Z$.   
    \end{enumerate}
\end{theorem}
The following theorem is the most technical statement in this paper and builds on several works by this author and A. G. Myasnikov in~\cite{MS, MS2010, MS2021, MS2024} and O. V. Belegradek in~\cite{beleg94}. The proof and all relevant notation and details are included in Sections~\ref{sec:charac} and \ref{sec:conv}. 

\begin{theorem}[Characterization Theorem]\label{thm:char} Let $A$ be an FDZ-ring where $\text{Spec}^0(P_{f_A})$ is connected. If $B$ is any ring $B\equiv A$, then
$B\cong (A(R),B_0, f,g,h)$
for some ring $R\equiv \Z$, some $B_0\equiv A_0$ and choice of cocycles $f,g$ and $h$ as described in Section~\ref{subsec:cocycles}. Conversely, any ring $(A(R),B_0,f,g,h)$ is elementarily equivalent to $A$.  \end{theorem}

\section{Bilinear maps and QFA property for FDZ-rings}\label{sec:bilin}
In this section, we discuss the notion of the bilinear map of a ring and the largest scalar ring of the bilinear
map. Towards the end of the section, we provide a proof of Theorem~\ref{thm:Main1}.

Consider a full non-degenerate bilinear map $f:M\times M \to N$ for some $R$-modules, $M$ and $N$. The mapping $f$ is said to have \textit{finite width} if there is a natural number $S$ such that for every $u\in
N$ there are $x_i$ and $y_i$ in $M$ we have
$$u=\sum_{i=1}^nf(x_i,y_i).$$
The least such number, $w(f)$, is the \textit{width} of $f$.

A set $E=\{e_1,\ldots e_n\}$ is a \textit{complete system}\index{complete system} for a non-degenerate mapping $f$ if $f(x,E)=f(E,x)=0$ implies $x=0$. The
cardinality of a complete minimal system for $f$ is denoted by $c(f)$.

We say that a mapping $f$ is of \emph{finite type} if both $w(f)$ and $c(f)$ are finite.

\begin{prop}[\cite{alexei86, alexei91}]\label{myasnik:bilinT} Let $f:M\times M \to N$ be a full non-degenerate bilinear mapping of finite-type. Then the largest scalar (commutative associative unitary) ring $P(f)$ with respect to which $f$ remains bilinear exists. Furthermore, the ring $P(f)$ is absolutely interpretable in $f$. The formulas of the interpretation depend only on $w(f)$ and $c(f)$. \end{prop}

The following two statements are basic facts  about Noetherian modules.
\begin{lemma}\label{f.g.module:lem} Assume that $R$ is Noethrian and $M$ and $N$ are finitely generated $R$-modules; then $Hom_R(M,N)$ is finitely generated as a $R$-module.\end{lemma}

\begin{cor}Assume $R$ is a Noetherian scalar ring and $M$ is a finitely generated $R$-module. Then $End_R(M)$ (the ring of endomorphisms of the $R$-module $M$) is f.g. as an $R$-module.\end{cor}

 We denote the quotient ring $A/Ann(A)$ by $\barA$, and for any $a\in A$ we denote $a+Ann(A)$ by $\bara$. Consider the full non-degenerate bilinear map:
$$f_A: \barA \times \barA\to A^2, \quad f(\barx,\bary)=xy, \forall x,y\in A$$

\begin{lemma}\label{lem:ft}
 Let $A$ be an $FDZ$-ring, and let $B$ be a ring. Then the following holds:
 \begin{enumerate}
     \item The ideal $A^2$ is absolutely definable in $A$, by an $\Lrings$-formula $\Theta$. Moreover, there exists an $\Lrings$-sentence $\Phi$ such that $A\models \Phi$ and if $B\models \Phi$, then $\Theta$ defines $B^2$ in $B$.
     \item $f_A$ is of finite type and $f_A$ is 0-interpretable in $A$. Moreover, there is a sentence $\Psi$ of $\Lrings$ such that $A\models \Psi$ and if $B\models \Phi\wedge \Psi$, then  $w(f_B)\leq w(f_A)$ and $c(f_B)\leq c(f_A)$. Therefore, the same formulas interpreting  $f_A$ in $A$ interpret $f_B$ in $B$. \end{enumerate}
\end{lemma}
\begin{proof} 

(1.) Let $\{a_1, \ldots, a_n\}$ be a set of generators of $A$ as an abelian group. For any $x $ and $y$ in $A$, there exist $z_1, \ldots, z_n$ in $A$, such that $xy=\sum_{i=1}^n a_iz_i$. Therefore, for any $m$ and any pairs $(x_i,y_i)$, $i=1, \ldots m$ there exist elements $s_i$, $i=1,\ldots n$ such that
$$\sum_{i=1}^m x_iy_i = \sum_{i=1}^n a_i s_i$$
Now, let 
$$\Phi=\forall x_1,y_1, \ldots, x_{n+1},y_{n+1} \exists s_1,t_1,\ldots, s_n,t_n (\sum_{i=1}^{n+1} x_iy_i = \sum_{i=1}^{n} s_it_i)$$
Clearly for any $B$ such that $B\models \Phi$, $B^2$ is defined in $B$ 
by formula $$\Theta(x)= \exists x_1,y_1, \ldots x_n,y_n (x=\sum_{i=1}^{n} x_iy_i)$$

 (2.) First, $Ann(A)$ is absolutely and uniformly definable in $A$. Hence $\hat{A}$ is absolutely and uniformly interpretable in $A$. For $a_i$ as in the proof of (1.), $E=\{\bara_1, \ldots, \bara_n\}$ is a finite complete system for $f_A$. It is easy to see that there exists a formula $\psi(\bar{x})$ such that if $ \psi(\bar{x})$ by a tuple $\bar{x}$, then $\bar{x}$ is a complete system of cardinality equal to the parity of $\bar{x}$. Let $\Psi=\exists \bar{x} \psi(\bar{x})$.  Then by part (1.) the statement is obvious. \end{proof}

\begin{definition}
 Let $A$ be an FDZ-ring. By $P(A)$ we denote the largest subring of $P(f_A)$ making the canonical homomorphism $\pi: A^2 \to \hat{A}$ $P(A)$-linear.
\end{definition}
We note that the subring $P(A)$ exists when $A$ is FDZ and it is a definable subring of $P(f_A)$.  

The following corollary is an immediate consequence of the above statements.

\begin{cor}\label{P(f)-interpret:cor} Let $A$ be an FDZ-ring, and consider the bilinear map $f_A$ and the scalar ring $P(f_A)$. Then

\begin{enumerate}
\item The ring $P(f_A)$ and its action on $\barA$ and $A^2$ are 0-interpretable in $A$. 
\item  The ring $P(A)$ is a definable subring of $P(f_A)$.
\item For any finitely generated ring $B\models \Phi \wedge \Psi$, as in Lemma~\ref{lem:ft}, the same formulas that interpret $P(f_A)$ (or $P(A)$) in $A$ interpret $P(f_B)$ (resp. $P(B)$) in $B$.\end{enumerate}
\end{cor}

We are now ready to provide a proof of Theorem~\ref{thm:Main1}.

\noindent \textit{Proof of Theorem~\ref{thm:Main1}.}

\noindent (1.)$\Rightarrow$ (2.): The same as the Proof of Theorem 16 in \cite{Rich}.\medskip

\noindent (2.) $\Rightarrow$ (1.): Consider $A$ and $\barA$ as above, and let $B$ be a finitely generated ring. Since $\barA$ is a finitely generated $P(f_A)$-module,
 there is an $\Lrings$-sentence, say $Gen$, which holds in $A$, such that, if $B\models Gen$, then $B$ is a f.g. $P(f_B)$-module. Assume $B\models Gen$. Then $\barB$ is a finitely generated $P(f_B)$-module. We now prove that $P(f_B)$ is a finitely generated ring. Note that $\barB$ is a finitely generated ring. Let $b_1,\ldots, b_m$ be a set of generators of $\barB$ as a $P(f_B)$-module, and let $B_1, \ldots, B_n$ be a set of generators of $\barB$ as a ring. For each $1\leq i\leq m$, and $1\leq j\leq n$ consider $\alpha_{ij}\in P(f_B)$ such that
 $$B_ib_j= \sum_{k=1}^n \alpha_{ij}b_k.$$
 Let $R$ be the unitary subring of $P(f_B)$ generated by $\alpha_{ij}$. We show that $P(f_B)$ is a finitely generated $R$-module. Firstly, $R\leq P(f_B)$. Since $R$ is finitely generated, it is a Noetherian ring. Furthermore, $End_R(\barB)$ is a finitely generated $R$-module. As $P(f_B)$ embeds in $End_R(\barB)$, it is a Noetherian $R$-module. Therefore, $P(f_B)$ is a finitely generated ring, since $R$ is a finitely generated ring. By \ref{AKNS1}, $P(f_A)$ is QFA. Hence, there exists a sentence $\Phi_2$ such that $A\models\Phi_2$, and if $B\models\Phi_2$, $P(f_A)\cong P(f_B)$. Therefore, $B$ is also a tame FDZ-ring. Now, an application of Theorem~\ref{MOS-4} finishes the proof of this theorem. \qed

\begin{remark}
    For most of the upcoming results, working with $P(f_A)$ suffices. When we need $P(A)$, for example, in Lemma~\ref{lem:6term}, it will be clear that in the statement, one can replace $P(f_A)$ with $P(A)$ without essentially changing the proofs. We do not need to distinguish between $P(f_A)$ and $P(A)$ and duplicate the statements. 
\end{remark}
\section{Some necessary and sufficient conditions for bi-interpretability with \texorpdfstring{$\Z$}{Z}}\label{sec:biint}
In this section, we prove Theorems~\ref{thm:main2} and \ref{thm:Main3}, specifying some necessary and sufficient conditions for the bi-interpretability of an FDZ-ring with $\Z$. \medskip

\noindent \textbf{Theorem~\ref{thm:main2}.} Let $A$ be an infinite $FDZ$-ring. If the annihilator $Ann(A)$ is infinite and $\Delta(A)\neq A$ then $A$ is not bi-interpretable with the ring of integers $\Z$.

\begin{proof} To derive a contradiction, assume that $A$ is bi-interpretable with $\Z$. By a theorem of \cite{Oger2006} (see also the proof of Theorem 16 in~\cite{Rich}), $Ann(A)\leq \Delta (A)$; otherwise, $A$ is not QFA and is not bi-interpretable with $\Z$. Since $A$ is an FDZ-ring, $Ann(A)$ is a finitely generated abelian group that is infinite by assumption. Let $a$ be an element of infinite order in $Ann(A)$. By assumption, $A/\Delta(A)$ is a non-empty free abelian group of finite rank. Choose $a\in A\setminus \Delta(A)$, where $a+\Delta (A)$ is a basis element of $A/\Delta(A)$. Note that $\langle a\rangle\cap Ann(A)\leq \langle a\rangle\cap \Delta(A)=\{0\}$, where $\langle a\rangle=\{za| z\in \Z\}$. Therefore, there exists a subgroup $B$ of $A$ such that $A=\langle a\rangle \oplus B$ are abelian groups. Set $Z=\langle a\rangle$. Since $A$ is bi-interpretable with $\Z$ the $\Z$-submodule $Z$ is definable in $A$. If, in the bi-interpretation, $\Z$ is interpreted in $A$ with constants, then $Z$ is definable with respect to these constants. Let us also choose an element of infinite order $b\in Ann(A) \leq B$. 

Let $(I,D)$ be an $\alpha_1$-incomplete ultrafilter on $I$. By $M^*$ we mean the ultrapower $M^I/D$ for a structure $M$. Note that the additive group of the ultrapower $\Z^*$ can be written as $\Z^*\cong S \oplus \Q^\delta$, where $\Q$ is the additive group of rational numbers,  $\delta$ is an infinite cardinal, and $S$ is reduced (has no divisible elements other than $0$); the copy of $\Z$ inside $\Z^*$ is inside $S$. Therefore, there are uncountably many distinct homomorphisms of $\Z^*$ that are zero on $S$, but restrict to nontrivial automorphisms of $\Q^\delta$. Pick one of such homomorphisms, say $\phi$. Since the additive group of $A$ is a finitely generated abelian group, there is a natural ring isomorphism $A^* \to A \otimes \Z^*$. Also, as abelian groups:
$$A^*\cong Z^* \oplus B^*\cong Z\otimes \Z^* \oplus B\otimes \Z^*$$
Therefore, there is a non-trivial \emph{ring} automorphism $\tilde{\sigma}:A^*\to A^*$ corresponding to the following:
$$\sigma:  Z\otimes \Z^* \oplus B\otimes \Z^* \to Z\otimes \Z^* \oplus B\otimes\Z^*$$
$$\left\{\begin{array}{ll} \sigma(ta) &= ta + \phi(t)b, \forall t\in \Z^* \\
\sigma|_{ B\otimes \Z^*}&=Id_{ B\otimes \Z^*}\end{array}\right.$$ 

Now consider the corresponding automorphism $\tilde{\sigma}: A^*\to A^*$. Observe that $Z^*$ is a definable subset of $A^*$ with respect to the diagonal images of the same constants with respect to which $Z$ is definable in $A$. By construction, $\tilde{\sigma}$ fixes (the diagonal copy of) $A$ in $A^*$, but $\tilde{\sigma}(Z^*)\not\subset Z^*$, contradicting the assumption that $Z^*$ is definable in $A^*$. This finishes the proof. \end{proof}

Next, we prove Theorem~\ref{thm:Main3}, providing two different sufficient conditions for bi-interpretability of an FDZ-ring with $\Z$.
\medskip

\noindent \textbf{Theorem~\ref{thm:Main3}.} Let $A$ be an FDZ-ring. If $A$ has a finite annihilator or $A=\Delta(A)$, and, in addition, $Spec(P(f_A))^0$ is connected, then $A$ is bi-interpretable with $\Z$.
\begin{proof} Assume $A$ is a FDZ-ring and $\Spec^\circ(P(f_A))$ is connected. Since $P(f_A)$ is interpretable in $A$ and $\Spec^\circ(P(f_A))$ is connected, by~\ref{AKNS2}, $\Z$ is bi-interpretable with $P(f)$, so it is interpretable in $A$. We refer to the corresponding coordinate map as $f:A\to \Z$. 

 First, we show that $A$ is interpretable in $\Z$. Let $\tau(A)$ denote the torsion subgroup $A$ which must be finite. In fact, $\tau(A)$ is an ideal of $A$. Since it is a direct sum of cyclic groups, $A$ splits over $\tau(A)$ as an abelian group. By assumption $Ann(A)$ is finite, so $Ann(A)\leq \tau(A)$. Therefore, a ring structure can be defined on $(A/\tau(A)) \times \tau(A)$ , which is isomorphic to $A$, as follows. Let $a_1, \ldots a_l$ be elements of $A$ whose canonical images in $A/\tau(A)$ form a free $\Z$ basis for $A/\tau(A)$ and pick $b_1, \ldots, b_m$ in $\tau(A)$, such that each $b_i$ is a generator of cyclic factor in the invariant factor decomposition of $\tau(A)$. Let $d_i$ be the order of $b_i$ for each corresponding $i$. Then there are unique fixed integers $c_{ijk}, 1\leq i,j\leq l$, $1\leq k \leq l$ and integers $0\leq t_{ijk}< d_k$,  $1\leq i,j\leq l$, $1\leq k \leq m$ such that
$$a_i\cdot a_j= \sum_{k=1}^l c_{ijk} a_k + \sum_{k=1}^m t_{ijk} b_k,~\forall 1\leq i,j\leq l.$$
Since $\tau(A)$ is an ideal of $A$, there are unique tuples of integers $\bar{s}_{ij},\bar{u}_{ji}, \bar{v}_{jj} \in \Z^m$ where for each $1\leq i\leq l$ and $1\leq j\leq m$:
\begin{align*} a_i\cdot b_j&=\sum_{k=1}^m s_{ijk} b_k\\
b_j\cdot a_i&=\sum_{k=1}^m u_{ijk} b_k\\
b_i\cdot b_j&=\sum_{k=1}^m v_{ijk}b_k \end{align*}

Observe that for each $d_i$ chosen above, $\Z_{d_i}=\{0,1,\ldots, d_i-1\}$ is 0-definable in the ring $\Z$ where addition and multiplication are defined appropriately modulo $d_i$. So, there exists a 0-interpretation $\Gamma$ of $A$ in $\Z$, interpreted in $\Z$ as the definable subset $B=\Z^l \times \Z_{d_1}\times \cdots \times \Z_{d_m}$, while the ring structure on $B$ is defined through the tuples of structure constants $\bar{c}, \bar{t}, \bar{s},\bar{u}$ and $\bar{v}$ as clearly specified above. This sets up an obvious interpretation of $A$ in $\Z$. We refer to the corresponding coordinate map as $g:B\to A$.

By Nies~\cite{Nies2007} it is left to prove that if $Ann(A)$ is finite or $\Delta(A)=A$, then there exists an injective map $A\hookrightarrow f^*(\Z)$ (where $f^*(\Z)$ is the copy of $\Z$ interpreted in $A$ by $f$) which is definable in $A$.
 
    First, assume that $Ann(A)$ is finite. Since $P(f_A)$ is bi-interpretable with $\Z$, the canonical copy of $\Z$ is definable in $P(f_A)$. This in turn implies that for any $a\in A$ of infinite additive order, $\langle a\rangle+Ann(A)$ is definable in $A$. Pick elements $a_1,\ldots , a_l$ and $b_1, \ldots, b_m$ from $A$ as in the previous paragraph. By the above, each $A_i=\langle a_i\rangle+Ann(A)$ is definable in $A$. Let $n$ be the exponent of $Ann(A)$. Then, for each $i$, $nA_i \subseteq \langle a_i\rangle$, so $\langle a_i\rangle$ is definable in $A$ as the subset $\{a_i, \ldots, (n-1)a_i\}\cup nA_i$. Now, the map
$A\to \Z^l \times \Z_{d_1}\times \cdots \times \Z_{d_m}$ sending $a\in A$ to the tuple $(c_1(a), \ldots, c_l(a),t_1(a), \ldots, t_m(a))$ where the $c_i$ and $t_i$ are the unique natural numbers (in the case of $t_i$'s unique mod the $d_i$'s) such that $a=\sum_{k=1}^lc_k(a) a_i + \sum_{k=1}^mt_k(a) b_k$ is a bijective map which is also definable in $A$. Therefore, $A$ is bi-interpretable with $\Z$  in this case.

Now, assume that $A=\Delta(A)$. Since $P(f)$ is bi-interpretable with $\Z$ and the action of $P(f)$ on $A^2$ is interpretable in $A$, for any element $a$ of infinite order in $A^2$ the infinite cyclic group $\langle a\rangle$ is definable in $A$. Since $A/A^2$ is a finite abelian group, this simply implies that for any $a\in A$, the subgroup $\langle a\rangle$ is definable in $A$. Similarly to the case above, we have an isomorphism $A\to \Z^l \times \Z_{d_1}\times \cdots \times \Z_{d_m}$ that is definable in $A$.

 \end{proof}   
  
\section{Arbitrary models of \texorpdfstring{$\Th(A)$}{Th(A)} when $A$ is super tame} \label{sec:A-super-tame}
In this section, we attempt a complete characterization of all models of the complete first-order theory of a super tame FDZ-ring $A$.

\begin{lemma}\label{lem:regular-biint} Every finitely generated structure $A$ bi-interpretable with $\Z$ with respect to constants $\bar{c}$ of $A$ is regularly bi-interpretable with $\Z$.\end{lemma}
\begin{proof} Let $\alpha(\bar{x})$ be the formula obtained in \cite{Nies2007}, Claim 7.15. First, any $\bar{c}$ as above satisfies such a formula. Assume $\bar{c}'$ is any other tuple such that $A\models \alpha(\bar{c}')$. By Lemma 2.30 of \cite{AKNS}, there exists a formula $\theta(\bar{x},\bar{y})$ in the language of $A$, such that $A\models \theta(\bar{c}, \bar{c}')$ if and only if the assignment $\bar{c}\mapsto \bar{c}'$ extends to an automorphism of $A$. Let $\tilde{c}$ and $\tilde{c}'$ be the images of $\bar{c}$ and $\bar{c}'$ in $A$ under the composition of interpretations, and let $\theta_{\Gamma\circ\Delta}$ be the translation of $\theta$ under the composition of interpretations. Then $A\models \theta_{\Gamma\circ\Delta}(\tilde{c},\tilde{c}')$ if and only if $\tilde{c}\mapsto \tilde{c}'$ extends to an isomorphism $(A,\bar{c})\to (\tilde{A},\tilde{c}')$.\end{proof}

\begin{lemma}\label{lem:commutative-Mod} Let $A$ be an $FDZ$-ring, and let $R$ be a scalar ring such that $R\equiv \Z$. Then $A\equiv R\otimes_\Z A$. \end{lemma}  

 \begin{proof}
 natural numbers $1\leq k \leq m$ and elements $a_i$, $1\leq i\leq m$, such that $A^+\cong \bigoplus_{i=1}^m \langle a_i\rangle$, each $a_i$ is a generator of the cyclic group $\langle a_i\rangle$, $a_i$ has a finite exponent if $1\leq i\leq k$, and it is torsion-free otherwise. Let $T\cong \bigoplus_{i=1}^k \langle a_i\rangle$ and $F\cong \bigoplus_{i=k+1}^m \langle a_i\rangle$. We note that $T$ is a torsion abelian group and $F$ is a free abelian group.

 \noindent \emph{Claim I.} The natural homomorphism $\phi:A\to B=R\otimes_\Z A$, $a\mapsto 1\otimes a$ is an injective homomorphism. In fact, if $\phi: A \to A'$ is an injective homomorphism of FDZ-rings, then $\phi\otimes_\Z R: A\otimes_\Z R\to A'\otimes_\Z R$ is an injective homomorphism.   
 
 \noindent \emph{Proof of Claim I.}  Consider $a_i's$, $T$ and $F$ as described above and assume that $\sum_{i=1}^k 1\otimes z_i a_i+ \sum_{i=k+1}^m 1\otimes z_i a_i=0$ in B for some $z_i\in \Z$. Since $F\otimes R$ is a free $R$-module, we conclude that $z_i=0$, $i=k+1, \ldots m$. Consider $1\leq i\leq k$. Since $1\otimes a_i$ are linearly independent, we must have $1\otimes z_ia_i=0$ in $B$. Since $R\equiv R$, $1$ is not divisible by $d_i$ in $R$. Therefore, $z_i\in d_iR\cap \Z=d_i\Z $. and $\sum_{i=1}^m z_i a_i+ \sum_{i=1}^m z_i a_i=0$ in $A$. The converse is clear.

     In fact, for each $1\leq i\leq k$, for each $1\leq i\leq k$, $R\otimes \langle a_i\rangle \cong R/d_iR$.
 Since $d_iR$ is defined uniformly in $R$,
 $$R\otimes \langle a_i\rangle \cong R/d_iR \equiv \Z/d_i\Z\cong \langle a_i\rangle$$
 and since $\Z/d_i\Z$ is finite $R\otimes \langle a_i\rangle \cong \langle a_i\rangle$. Therefore, the subring of $B$ generated by $a_i$, where $1\leq i\leq m$, is isomorphic to $A$.\qed

Going back to the proof of the lemma, we note that $B=R\otimes_{\Z}A$ is 0-interpretable in $R$ using the same formulas that 0-interpret $A$ in $\Z$. To be more precise, the formula that defines $ \Z_{d_1}\times \cdots \times\Z_{d_k}\times \Z^l$ in $\Z$ defines $M=\Z_{d_1}\times \cdots \times \Z_{d_k}\times R^l$ in $R$. However, this base set is isomorphic to $B$ as an $R$-module. The multiplicative structure in $M$ can now be defined using the integer structure constants that define the multiplicative structure in $A$ in terms of the basis elements $a_i$. This makes $M$ isomorphic to $B$ as rings, so that is the desired interpretation of $B$ in $R$. Since $\Z\equiv R$, and the above interpretations are uniform and parameter-free we can conclude that $A\equiv B=R\otimes_\Z A$. \end{proof}
\begin{theorem} Let $A$ be a super tame FDZ-ring, and assume that $B$ is a ring. Then the following are equivalent:
\begin{enumerate}
    \item $A\equiv B$
    \item $B\cong R\otimes_{\Z}A$ for some scalar ring $R\equiv \Z$.   
    \end{enumerate}
\end{theorem}
\begin{proof} (2.) $\to$ (1.) was already proven in the previous lemma. To prove the other direction, we first note that $A$ is regularly bi-interpretable with $\Z$. Therefore, there is a first-order sentence that describes the structure of $A$ as a definable subset of $\Z$ with respect to any parameters $\bar{c}$ satisfying a formula recovered in Lemma~\ref{lem:regular-biint}. The same sentence will describe an interpretation $B$ in a model $R$ of a fragment of $Th(\Z)$. Since the interpretation is regular and $B$ satisfies any true sentence in $A$, then $R\equiv \Z$. Therefore, $B\cong \Z_{d_1}\times \cdots \times \Z_{d_k}\times R^l$ as an $R$-module, while the $R$-algebra (ring) structure on $B$ is defined by the same integer structure constants that define the structure of $A$. By the previous lemma, we have $B\cong R\otimes_{\Z} A$.
    
\end{proof}

\section{Characterization of the arbitrary models of \texorpdfstring{$\Th(A)$}{Th(A)} when \texorpdfstring{$P(f_A)$}{P(fA)} is bi-interpretable with \texorpdfstring{$\Z$}{Z}}\label{sec:charac} 
In this section, we provide a proof of one direction (the characterization direction) of the main theorem of this paper: Theorem~\ref{thm:char}. 
\subsection{Abelian deformations of tensor completions of FDZ-rings over non-standard models of \texorpdfstring{$\Z$}{Z}}
Consider an FDZ-ring $A$ and the ideals $K(A)$ and $L(A)$ introduced earlier in the paper.

Let $R\equiv \Z$ and let $A$ be an FDZ-ring. Consider $B=R\otimes_\Z A$. In general, $A$ might have an addition $A_0$. Let $A(R)=R\otimes A$. In Claim I of the proof of Lemma~\ref{lem:commutative-Mod}, we proved that since
$$0\to Ann(A) \to A \to \barA\to 0$$ is a short exact sequence of FDZ-rings, then for $R\equiv \Z$, the following is a short exact sequence of $R$-algebras.
$$0\to Ann(A)\otimes_\Z R \to A\otimes_\Z R \to \barA\otimes_\Z R\to 0$$ 
where $Ann(A(R))=R\otimes Ann(A)$, and $A_0\otimes R$ is, in fact, an addition $(A(R))_0$ of $A(R)$.

\begin{lemma}\label{lem:6term}
    Let $A$ be an FDZ-ring where $\text{Spec}^0(P(f_A))$ is connected, and let $B$ be a ring that is elementarily equivalent to $A$. Then there exists a scalar ring $R\equiv \Z$ such that the following is true:
    \begin{enumerate}
 \item  $Ann(B)$ fits into the following short exact sequence of abelian groups
    $$ 0\to O(A(R)) \to Ann(B)\to B_0\to 0$$
    where $O(A(R))=Ann(A(R))\cap \Delta(A(R))$ and $B_0\equiv A_0$ are abelian groups.
\item There exist isomorphisms, $\eta, \psi, \phi$ and $\mu$ of $R$-algebras: 
that make the following diagram commutative:

$$\begin{tikzcd}[column sep=small]
0 \arrow[r] & O(A(R)) \arrow[r] \arrow[d,"\eta"] &
\Delta(A(R)) \arrow[r] \arrow[d,"\psi"] &
\widehat{A(R)} \arrow[r] \arrow[d,"\phi"] &
A/K(A(R)) \arrow[r] \arrow[d,"\mu"] & 0 \\
0 \arrow[r] & O(B) \arrow[r] &
\Delta(B) \arrow[r] &
\widehat{B} \arrow[r] &
B/K(B) \arrow[r] & 0
\end{tikzcd}$$

where the horizontal arrows represent the restrictions of the obvious canonical mappings.
   \item $B$ fits into the following short exact sequence of rings
    $$0\to Ann(B) \to B \to \widehat{A(R)}\to 0$$

    \end{enumerate}
     
\end{lemma}
\begin{proof}
  Since the ideals $Ann(A)$ and $\Delta(A)$ are uniformly 0-definable in A, and $P(f_A))$ is bi-interpretable with $\Z$, we have that $O(B)=Ann(B)\cap \Delta(B)$ is naturally isomorphic to its tensor extension $O(A(R))\cong O(A)\otimes R$ for some $R\equiv \Z$ (see the previous section for complete arguments in a similar case). Also, since $f_A$ is bilinear with respect to $P(f_A)$, we note that $f_B$ is bilinear with respect to this ring $R$. Moreover, $$A_0\cong Ann(A)/O(A)\equiv Ann(B)/O(B)\cong B_0$$  since all the definitions and required interpretations are uniform. This concludes the proof of (1.).

Proof of (2.) also follows the same path as the Proof of (1.). Exactness of the sequences is clear. Since $\Z$ is bi-interpretable with $P(f_A)$, the two-sorted structure $\langle \Z, f_A\rangle$ is interpretable in $A$ uniformly. Therefore, $\langle R, f_B\rangle$ is interpretable in $B$ with the same formulas, hence $\Delta(B)\cong R\otimes \Delta(A)$ and $\widehat{B}\cong \widehat{A(R)}$. Since the isomorphisms are obtained from the action of $P(f_B)$ on $\widehat{B}$ and $\Delta(B)$, and by definition $P(f_B)$ makes the restriction $\Delta(B) \to \widehat{B}$ $P(f_B)$-linear, the middle square must be commutative. The other two squares are commutative as a consequence.

Statement (3.) is a corollary of (1.) and (2.).

\end{proof}

Next, we wish to describe the ring structure of a ring $B\equiv A$ in terms of the ring $A(R)$, when $R$ is the scalar ring $R\equiv \Z$ recovered in Lemma~\ref{lem:6term} using insights from the same lemma.

\subsection{Construction of abelian deformations using symmetric 2-cocycles} \label{subsec:cocycles}
In the following lines, we try to illuminate the structure of $B\equiv A$ by deforming the structure of $A(R)$ via symmetric cocycles based on the information we gained in Lemma~\ref{lem:6term}. For details on extension theory and its relationship to the second cohomology group, we refer the reader to~\cite{robin}. Here, we only need to understand abelian extensions and their relationship with the functor $\Ext$. In this subsection, we keep the discussion somewhat informal. We shall prove the necessary statements in the next subsection.

Consider  $B$ as in Lemma~\ref{lem:6term} and the group $B_0$ as described in (1.) of the same lemma. Let $T$ be the torsion subgroup of $O(A(R))$, which is a finite abelian group. Hence, there exists a torsion-free subgroup $F_1$ of $O(A(R))$ such that $O(A(R))\cong T\oplus F$. Now 
$$\Ext(B_0, O(A(R)))\cong \Ext(B_0, F\oplus T)\cong \Ext(B_0, F)\oplus \Ext(B_0, T) \cong \Ext(B_0, F).$$. Therefore, the structure of $Ann(B)$ can be described as the set $B_0\times O(A(R))$, with the addition defined by
$$(b_1,o_1)\boxplus(b_2,o_2)=(b_1+b_2, o_1+o_2+h(b_1,b_2))$$
where $b_i\in B_0$, $o_i \in O(A(R))$, and some $h\in S^2(B_0, F)$. It is known that $\Ext(B_0,F)\neq 0$.

The cocycle $h$ recovered above also defines an extension of $\Delta(A(R))$ by $B_0$, which, for obvious reasons, we call $K(B)$. We also denote the copy of $\Delta(A(R))$ in $K(B)$ by $\Delta(B)$, again for obvious reasons. 

Next, recall that  $M=M(A(R))=A(R)/L(A(R))$ and $N=N(A(R))=L(A(R))/K(A(R))$. Since $N$ is a finite abelian group and $M(A(R))$ is torsion-free, there is a subgroup $A_1\cong M(A(R))$ of $A(R)/K(A(R))$ such that $A/K(A(R) = A_1 \oplus N(A(R))$ as an internal direct sum. In fact, $A_1$ can become a free $R$-module of the same rank as the free abelian group $M$. To avoid introducing a new notation, we assume $M(A(R))=A_1$. 

Consider two types of symmetric 2-cocycles:

\begin{enumerate}
\item $f\in S^2(M(A(R)), D)$
\item $g\in S^2(N, D)$ 
\end{enumerate}
There are multiple constraints that one needs to impose on the cocycle $g$, simply because $N=L(A)/K(A)$ is a finite uniformly interpretable quotient and both $(L(A))/\Delta(A)$ and $K(A)/\Delta(A)$ are uniformly interpretable quotients isomorphic to a free abelian group $A_0$. We elaborate on these constraints in the following lines.

Consider the following extension in $A(R)$ 
$$0\to K(A(R))\to L(A(R)) \to N\to 0$$
 Let $q\in S^2(N, K(A(R))$ be a cocycle that defines this extension up to equivalence. Let $\pi:A(R) \to \widehat{A(R)}$ be the canonical epimorphism. The homomorphism $\pi$ induces a homomorphism: $$\pi_*: \Ext(N, K(A(R)) \to \Ext(N, \pi(K(A(R))).$$ Identifying $\widehat{A(R)}$ and $\widehat{B}$ via the isomorphism $\phi$ and considering the canonical epimorphism $\pi': B \to \widehat{B}$, there exists a cocycle $p'\in \Ext(N, K(B))$ such that $\pi'_*(p)=\pi_*(q)$. It is also clear that for any choice of $g\in S^2(N, D)$, assuming that $D=Ann(B)$, we have $p+g \in S^2(N, K(B))$ and $\pi_*'(p+g)=\pi_*'(p)=\pi_*(q)$. 

As mentioned above, there is another important restriction that we need to impose on the cocycle $p+g$ or $g$ for that matter. We require that $L(B)/\Delta(B)$ is a torsion free abelian group elementarily equivalent to $K(B)/\Delta(B)\cong B_0$; that is, both are abelian groups elementarily equivalent to the free abelian group $A_0$. 
Unfortunately, we cannot provide an algebraic criterion here. Although we know that Szmielew’s Invariants are the same for elementarily equivalent abelian groups, these invariants do not necessarily define the structure of an abelian group unless it is finitely generated. In this case, there is an algebraic criterion based on Smith Normal Form and the minors of the relation matrix. Here, we deal with arbitrary models. However, we can impose a similar and very important restriction on $g$ as follows.  Since $N$ is finite and all the ideals involved are uniformly definable,  $$L(A(R)^*)/K(A(R)^*)\cong N^*\cong N\cong L(A)/K(A).$$ Pick $a_1, \ldots, a_m \in L(A)$ whose cosets in $L(A)/K(A)\cong N$ provide the primary factor decomposition of $N$, where the period of each $a_i$ is modulo $K(A)$. By the structure theorem for finitely generated abelian groups, $a_i$ can be chosen so that $\{e_ia_i+\Delta(A):i=1,\ldots, m\}$ can be extended to a basis of $Q(A)=K(A)/\Delta(A)$. This statement is not necessarily first order. However, we can express that for any number $d\geq 2$, there are elements in $a_1, \ldots, a_m$ in $L(A)$ whose images provide a primary decomposition for $L(A)/K(A)$ with periods $e_1|e_2|\cdots|e_m$, respectively, and the images of $\{e_ia_i+\Delta(A)|i=1,\ldots,m\}$ in $Q/dQ$ are linearly independent.  The same sentence will imply in $B$ that there are elements  $b_1, \ldots, b_m$ of $L(B)$ which generate $L(B)/K(B)$ and the images of $e_ib_i+\Delta(B)$ in $B_0/dB_0)\cong A_0/dA_0$ are independent. By Lemma~\ref{cocycle:lem} and the fact that $\Ext(\oplus_{i=1}^m G_m,D)\cong \oplus_{i=1}^m\Ext( G_m,D)$, the above will give a restriction on the definition of $g$.     

Now, define a new structure on the set $B=(M(A(R))\oplus N) \times K(B)$, where $\oplus$ is the direct sum of the groups and $\times$ is simply the Cartesian product of the sets, using the cocycles $f$ and $g$ as follows: 
 
Any $b\in B$ can be written as $$b=(a_1,a_2,k),$$ where $a_1\in M(A(R))$, $a_2\in N$, and $k\in K(B)$. Now, define an additive structure on $B$ by 
\begin{align*} (a_1,a_2,k) &\boxplus (a_1',a_2',k')\\
&=(a_1+a_1', a_2+a_2', k_1+k_1'+f(a_1,a_1')+p(a_2+a_2')+g(a_2+a'_2))\end{align*}
The above abelian group structure is well defined, since it is constructed as an abelian extension of two well defined abelian groups.

To define a ring product on $B$, we can identify $B/D$ with $A(R)/Ann(A(R))$. Also, there is a copy of $\Delta(A(R))$ in $K(B)$ as an additive subgroup, which we will formally call $\Delta(B)$ for now. Let us pick $b, b' \in B$ and let $a$ and $a'$ be preimages of $b+Ann(A(R))$ and $b'+Ann(A(R) \in \widehat{A(R)}$, respectively. Then $a\cdot a'$ is in $\Delta(B)$, where $\cdot$ is the product in $A(R)$. So we define
 $$b\boxdot b'=a\cdot a'$$

 The structure defined above is denoted by $B=(A(R),B_0,f,g,h)$.
\subsection{Proof of the Characterization Theorem}
\begin{definition}[Abelian deformations of $A(R)$]
Let $A$ be a FDZ-ring, let $R$ be a model of  $Th(\Z)$, and let $A(R)=R\otimes_\Z A$. Assume $B_0$ is an abelian group elementarily equivalent to an addition $A_0$ of $A$. The structure $(A(R),B_0,f,g,h)$ described in Subsection~\ref{subsec:cocycles} is called an abelian deformation of $B=A(R)$ for some scalar ring $R\equiv \Z$.
\end{definition}
\begin{lemma} The structure $B=(A(R),B_0,f,g,h)$ is a well-defined ring.\end{lemma}
\begin{proof} As already mentioned in Section~\ref{subsec:cocycles}, the addition is well defined since we are just defining an abelian extension:
$$0\to K(B)\to B \to A(R)/K(A(R)) \to 0$$
via the symmetric 2-cocycles $f$ and $g$, where the abelian group $K(B)=\Delta(A(R))+B_0$ is also a well-defined group as an abelian extension of $\Delta(A(R)$ by $B_0$ defined by the cocycle $h$. 

The product is well-defined modulo $D$. Any $b\in D$ annihilates all elements of $B$.  We just need to check bi-linearity. We check it on the left:

\begin{align*} d\boxdot(b\boxplus b')&= d\cdot  (b+b'+f(b_1,b_1')+g(b_2+b'_2))\\
&=d\cdot b +d\cdot b' +d\cdot(f(b_1,b_1')+g(b_2+b'_2)) \\
&=d\cdot b +d\cdot b'\\
&=d\cdot b \boxplus d\cdot b'~ (\textrm{since $d\cdot b, d\cdot b' \in R^2\subset \Delta(A(R))$}) \\
&=d\boxdot b \boxplus d\boxdot b'\end{align*}
\end{proof} 

\begin{theorem}[Characterization Theorem] Let $A$ be an FDZ-ring where $\text{Spec}^0(P(A))$ is connected. If $B$ is any ring $B\equiv A$, then
$B\cong (A(R),B_0, f,g,h)$
for some ring $R\equiv \Z$, some $B_0\equiv A_0$, and choice of cocycles $f,g$ and $h$ as described in Section~\ref{subsec:cocycles}.   
\end{theorem}
\begin{proof}
    We have already proven that $B\equiv A$ satisfies conditions (1.)-(3.) of  Lemma~\ref{lem:6term}. We shall refer to the maps and constructions mentioned in that lemma. Let us denote $(A(R),B_0,f,g,h)$ by $C$, where $f$, $g$ and $h$ are yet to be defined in terms of those defining the structure of the ring $B$. By construction, there is a cocycle $h'$ that defines the lower short exact sequence of the following diagram. Therefore, there exists a cocycle $h\in S^2(B_0, O(A(R))$ defining the upper exact sequence which induces an isomorphism $Ann(C)\to Ann(B)$ making the following diagram commutative.
$$\begin{tikzcd}[column sep=small]
0 \arrow[r] & O(A(R)) \arrow[r]  \arrow[d,"\eta"] &
Ann(C) \arrow[r] \arrow[d]&
B_0 \arrow[r]\arrow[d,"id"] &0\\
0 \arrow[r] & O(B) \arrow[r] &
Ann(B) \arrow[r]  &
B_0 \arrow[r]  & 0
\end{tikzcd}$$
Since $\Delta(C)=\Delta(A(R))$ and the isomorphism $\psi: \Delta(A(R))\to \Delta(B)$ agrees with the isomorphism $\eta$ and $K(B)=\Delta(B)+Ann(B)$, in fact $K(B)$ is defined as an extension of $\Delta(B)$ by $B_0$ using the cocycle $h'$, therefore, the cocycle $h$ defines $K(C)$ in a manner that there exists an isomorphism $\theta:K(C) \to K(B)$ making the following diagram commutative:

$$\begin{tikzcd}[column sep=small]
0 \arrow[r] & \Delta(A(R)) \arrow[r]  \arrow[d,"\psi"] &
K(C) \arrow[r] \arrow[d, "\theta"]&
B_0 \arrow[r]\arrow[d,"id"] &0\\
0 \arrow[r] & \Delta(B) \arrow[r] &
K(B) \arrow[r]  &
B_0 \arrow[r]  & 0
\end{tikzcd}$$

We note that $C/K(C)$ can be identified with $A(R)/K(A(R))$ and the isomorphism $\mu:C/K(C) \to B/K(B)$ respects the maximal torsion subgroup $N$. Therefore, $B/K(B)\cong \mu(N) \oplus M(B)$, where $M(B)\cong M(A(R))=M(C)$ as free $R$-modules. If one picks an $R$-basis, $m_1, \ldots ,m_k$ for $M(C)$ and the corresponding basis $\mu(m_1), \ldots \mu(m_k)$, then the inverse images of these under the corresponding canonical epimorphisms $\pi:C\to\widehat{C}$ and $\pi':B\to \widehat{B}$ generate subgroups $M_1(C)$ and $M_1(B)$ of $C$ and $B$, respectively, which contain the corresponding annihilators. Therefore, given a cocycle $f'\in S^2(M(B), Ann(B)$, defining the extension of in the lower row of the following diagram, there exists a cocycle $f\in S^2(M(C), Ann(C)$ defining the extension in the upper row, and an isomorphism $\theta_1$ making the diagram commutative:
$$\begin{tikzcd}[column sep=small]
0 \arrow[r] & K(C) \arrow[r]  \arrow[d,"\theta"] &
K(C)+M_1(C) \arrow[r] \arrow[d, "\theta_1"]&
M(C) \arrow[r]\arrow[d,"\mu|_{M(C)}"] &0\\
0 \arrow[r] & K(B) \arrow[r] &
K(B)+M_1(B) \arrow[r]  &
M(B) \arrow[r]  & 0
\end{tikzcd}$$
Since we have 
\begin{equation}\label{eqn:exts}
\begin{split}
\Ext(B/K(B),K(B))&\cong \Ext(M(B)\oplus N(B), K(B))\\
&\cong  \Ext(M(B), K(B))\oplus  \Ext(N(B), K(B))
\end{split}
\end{equation}

it is left to consider the extension:
$$0\to K(B)\to L(B) \to N(B)\to 0$$
Since $N(B)$ is the maximal torsion subgroup of $B/K(B)$, the cocycle, say $g'$, whose class in $\Ext (N(B), K(B))$ partially determines $L(B)$ will produce the subgroup $L(B)=Is(Ann(B) + B^2)$. Consider the corresponding extension in $A(R)$:
$$0\to K(A(R))\to L(A(R)) \to N\to 0$$
and let $q\in S^2(N, K(A(R))$ be the cocycle defining this extension up to equivalence and
Let $\pi:A(R) \to \widehat{A(R)}$ be the canonical epimorphism and $\pi': B \to \widehat{B}$ the one for $B$. The homomorphism $\pi$ induces a homomorphism: $$\pi_*: \Ext(N, K(A(R)) \to \Ext(N, \pi(K(A(R))).$$ Identifying $\widehat{A(R)}$ and $\widehat{B}$ via the isomorphism $\phi$, there exists a cocycle $p'\in \Ext(N(B), K(B))$ such that $\pi'_*(p')=\pi_*(q)$. It is also clear that for any choice of $g'\in S^2(N(B), Ann(B))$, we have $p'+g' \in S^2(N, K(B))$ and $\pi_*'(p'+g')=\pi_*'(p')=\pi_*(q)$. Finally, it is clear that there exists a cocycle $p\in S^2(N, K(C))$, and some cocycle $g\in S^2(N, Ann(C))$ such that $\pi''_*(p+c)=\pi'_*(p'+g')$, where $\pi'':C\to \widehat{C}$ is the canonical epimorphism and an isomorphism $\theta_2: L(C) \to L(B)$, making the following diagram commutative
$$\begin{tikzcd}[column sep=small]
0 \arrow[r] & K(C) \arrow[r]  \arrow[d,"\theta"] &
L(C) \arrow[r] \arrow[d, "\theta_2"]&
N \arrow[r]\arrow[d,"\mu|_{N}"] &0\\
0 \arrow[r] & K(B) \arrow[r] &
L(B) \arrow[r]  &
N(B) \arrow[r]  & 0
\end{tikzcd}$$

This finishes the proof considering the isomorphism in Equation~\ref{eqn:exts} above.
\end{proof}

\section{The converse of the characterization theorem}\label{sec:conv}
In this section, we prove the converse of our characterization theorem. The argument uses standard ultrapower constructions. The construction goes through two stages of taking ultrapowers.

\begin{lemma} Let $A$ be an $FDZ$-ring, $(I,D)$ a non-principal ultrafilter, and $A^*$ denote the corresponding ultrapower of $A$. Then $A^*\cong A(R)$, where $R=\Z^*$.\end{lemma}
\begin{proof} Pick the elements $a_i$ of $A$ as in Claim 1 in the proof of Lemma~\ref{lem:commutative-Mod}. The $(\bigoplus_{i=1}^m \langle a_i\rangle)^*\cong \bigoplus_{i=1}^m \langle a_i\rangle^*$. Each $\langle a_i\rangle^*$ has an additive group isomorphic to $R^+$ or $\Z/d_i\Z$. Both are cyclic $R$-modules. The ring structure is induced as an $R$-algebra structure defined using the same integer structure constants that define the product in $A$ in terms of elements $a_i$. Therefore, there is a natural isomorphism between $A^*$ and $A(R)$.\end{proof}
\begin{lemma} Let $B\equiv A$, where $A$ is an FDZ-ring, and let $(I,\D)$ be a non-principal ultrafilter. Using the identification of $B$ with its diagonal image in $B^*$, we have the following:
\begin{enumerate}
  \item[(a)] $Ann(B^*)=Ann(B)^*$
  \item[(b)] $(B^*)^2=(B^2)^*$
  \item[(c)] $\Delta(B^*)=(\Delta (B))^*$
\end{enumerate}
\end{lemma}

 \begin{proof} Statement (a) is clear. For (b), we first understand that the elements of the Cartesian power $B^I$ are functions $x: I \to B$. Now, the inclusion $\geq$ follows from the fact that $(B^I)^2$ is generated by $x\cdot y=z$, $x, y\in G^I$, where $z(i)=x(i)\cdot y(i)\in B^2$ for $\D$-almost every $i\in I$. But then, the equivalence class of $x\cdot y$ is in $(B^2)^*$. The other inclusion follows from the fact that $B^2$ is of finite width. Indeed, let $z\in G^I$ be such that the equivalence class $\tilde{z}\in (B^2)^*$ and assume that the width of $B^2$ is $n$, that is, every element can be written as a sum of at most $n$ elements of the form $x\cdot y$, $x,y\in B$. Then for $\D$-almost every $i\in I$, $z(i)=\sum_{k=1}^n x_k(i)\cdot y_k(i)$. Define $z_k\in B^I$, $k=1,\ldots n$ by $z_k(i)=x_k(i)\cdot y_k(i)$. Obviously $\widetilde{z_k}\in (B^*)^2$, $k=1, \ldots ,n$, and $\tilde{z}=\widetilde{z_1}+\cdots+ \widetilde{z_n}$. This implies the result.      

For (c)
\begin{align*}
\tilde{x}\in Is((B^2)^*)&\Leftrightarrow n\tilde{x}\in (B^2)^*, \text{ for some $n\in \N^+$}\\
&\Leftrightarrow \widetilde{nx}\in (B^2)^*,  \text{ for some $n\in \N^+$}\\ &\Rightarrow \tilde{x}\in (Is(B^2))^*.   
\end{align*}
If $\tilde{x}\in (Is(B^2))^*$, then $n(i)x(i)\in B^2$ for some $n(i)\in \N^+$ for $\D$-almost all $i\in I$. But then $mx(i)\in B^2$ where $m$ is the exponent of $Is(B^2)/B^2$, which implies that $\tilde{x}\in Is((B^2)^*)$.
\end{proof}
\begin{lemma}[Exercise 11.1.5 of~\cite{robin}]\label{cocycle:lem}
Let 
\[
0\to D \xrightarrow{\mu} E \xrightarrow{\varepsilon} G\to 0
\]
be a group extension with abelian kernel $ D$, and let $ G = \langle a \rangle $ be cyclic of order $ e$. Let $a= \varepsilon(x) $ with \( x \in E \). A transversal function \( \tau: G \rightarrow E \) is defined by

\[
\tau(ia) = ix \quad \text{for } 0 \leq i < e.
\]

Then the values of the corresponding cocycle $g$ are

\[
g(ia, ja) = 
\begin{cases} 
0 & \text{if } i + j < e, \\
d & \text{if } i + j \geq e,
\end{cases}
\quad \text{where } d = ex \in \ker \varepsilon.
\]
\end{lemma}

\begin{lemma}\label{lem:mono} Let $B=(A(R),B_0,f,g,h)$, and let $(I,\D)$ be a non-principal ultrafilter such that $B_0^*\cong (A(R)_0)^*$. Then, for every natural number $e\geq 2$, there exists a natural number $d$ that is relatively prime to $e$, and a monomorphism of rings $\theta: A(R)^*\to B^*$ such that $[B^*:\theta(A(R)^*)]=d$.\end{lemma}
\begin{proof} 
Consider the sequence of ideals in $A$: $$0\leq O(A) \leq Ann(A) \leq K(A) \leq L(A) \leq A$$
Since $A$ is FDZ, there are natural numbers $$m_0=0\leq m_1 \leq m_2 \leq m_3\leq m_4 \leq m_5=m$$ and elements $x_1,x_2, \ldots, x_m$ such that $\{x_{m_i+1}, \ldots, x_{m_{i+1}}\}$ fall in the $i$'th gap in that sequence, $i=1,\ldots 5$, and their images modulo the $i-1$'th ideal in the gap provide the primary factor decomposition for the corresponding quotients of the ideals. For example, $x_{m_1+1}+O(A), \ldots , x_{m_2}+O(A)$ produces a primary factor decomposition for $Ann(A)/O(A)\cong A_0$. Since this specific gap is special in what is coming, we set
$$n=m_2-m_1$$
For simplicity, we will also assume 
$$n=m_4-m_3$$
meaning that in the primary factor decomposition for $L(A)/K(A)$ we have $n$-terms, each cyclic factor of order $e_i$, $i=1,\ldots n$. We can make this assumption since $K(A)/\Delta(A) \cong L(A)/\Delta(A)\cong \Z^n$. So, in general, we  might assume some $e_i=1$. From the setup of the lemma, we can assume that the same sequence will generate $A(R)$ as an $R$-module while the cyclic factors are cyclic $R$-modules. If a cyclic factor is finite, as explained previously, it will remain the same since $R\equiv \Z$. Again, since we shall work with the above special gap frequently in the sequel, we pick nicknames for those $x_i$, and refer to the elements freely by either name when it suits the purpose. Set:
\begin{itemize}
    \item $a_i= x_{m_1+i}$, $i=1,\ldots ,n$
    \item $b_i=x_{m_3+i}$, $i=1, \ldots, n$
\end{itemize}
By the structure theorem for finitely generated abelian groups, we can assume that
$$e_ia_i \equiv  b_i, (\text{mod~} \Delta(A_i))$$

Note that $B^*\cong (A(R^*),B_0^*,f^*,g^*,h^*)$, where each of $f^*$, $g^*$, and $h^*$ is induced by $h$, $g$, and $h$, respectively. Since $M(A(R))^*$ and $B_0^*$ are torsion-free and saturated, $\Ext(B_0,O(R)^*)=0$; this, in turn, implies that $\Ext(A_1^*,D^*)=0$. This means that both $f^*$ and $h^*$ are coboundaries. Thus, we assume that they are trivial. We recall the discussion in Subsection~\ref{subsec:cocycles}. Let $g^*\in S^2(N, (R^*)^n)$. Let $N=G_1\oplus\cdots \oplus G_n$ be the primary decomposition of $N$, where $G_i\cong \Z/e_i\Z$. By~(\cite{fuchs}, Theorem 52.2), there is a short exact sequence:
$$0\to(R^*)^n\xrightarrow{e_i} (R^*)^n \to \Ext(\Z/e_i\Z,(R^*)^n)\to 0$$
Where the map $e_i$ is just multiplication by $e_i$, implying the existence of natural isomorphisms
$$\Ext(\Z/e_i\Z, (R^*)^n)\cong (R^*)^n/e_i(R^*)^n\cong \Z^n/e_i\Z^n$$
In view of the above and Lemma~\ref{cocycle:lem}, the cocycle $g^*$ is cohomologous to a cocycle defined under the diagonal embedding of $\Z\to R\to R^*$. Therefore, the following holds in $B$ and $B^*$:
\begin{enumerate}
    \item there are elements $u_i \in L(B)$ such that, where $u_i+\Delta(B^*)$ is a $\Z^*$-basis of $L(B^*)/\Delta(B^*)$
    \item The elements $e_iu_i+\Delta(B^*)$ lie in $B_0^*$, We observe that some $e_i$'s could be 1.
    \item There is a $\Z^*$ basis $v_1, \ldots,  v_n$ of $B_0^*$ and numbers $d_1, \ldots, d_n$, such that $d=d_1\cdots d_n$ is relatively prime to $e$ and $e_iu_i=d_iv_i$ (mod $\Delta(B^*)$) for each $i$.
    \end{enumerate}
Going back to the elements $x_1, \ldots, x_m$ we picked in $A(R)$, we can select a similar set of elements $y_1,\ldots , y_m$ in $B^*$, where $u_i$'s function similarly to $a_i$'s and the $v_i$ will be similar to the $b_i$.  
    
Since $h^*$ is a coboundary, there is a subgroup of $K(B^*)$ isomorphic to $B_0^*$ that splits from $K(B^*)$. We call it, for simplicity, again by $B_0^*$. Hence, $K(B^*)\cong \Delta(B^*) \times B_0^*$. It is immediately clear that there exists an isomorphism $(A(R)^*)/(A(R)^*_0)\to B^*/B_0^*$ restricting to an isomorphism that takes the canonical image of $\Delta(A(R)^*)$ onto that of $\Delta(A(R)^*)$ and maps $a_i+\Delta(A(R)^*) \mapsto u_i+\Delta(B^*) $. Pick a transversal of the canonical homomorphism $A(R)^* \to A(R)^*/(A(R)^*_0)$ that corresponds to the choice of $a_i$ above and remains the identity on $\Delta(A(R)^*)$. Let $A(R)^*_f$ be the subring of $A(R^*)$ generated by this transversal. Clearly, $A(R)^*_f$ and $A(R)^*_0$ generate $A(R)^*$. Similarly, consider a transversal of $B^* \to B^*/B^*_0$ that corresponds to the choice of $u_i$ above and remains the identity on $\Delta(B^*)$. Let $B^*_f$ be the subring of $B^*$ generated by this transversal. Define a map $\phi_e:A(R)^*_f\to B^*_f$, extending $x_i\mapsto y_i$ to an $R^*$-module map, $i\neq m_3+1, \ldots m_4$, that is, $\phi_e$ induces our given isomorphism  $(A(R)^*)/(A(R)^*_0)\to B^*/B_0^*$ and also $\phi_e(a_i)=u_i$ for all $i=1,\ldots n$. Since the cocycle $g^*$ defines $(B^*)_f$ as an extension of $(B_0^*)\cap B^*_f$ by $B^*/B_0^*$, the above is, in fact, a homomorphism. It is also 1-1 and onto by construction. Next, define $\psi_d:A(R)^*_0 \to B_0^*$ as a $\Z^*$-module monomorphism that extends to $b_i\to d_iv_i$. Note that $A(R)^*=A(R)^*_f + A(R)_0^*$ and $B_0^*= B_f^*+ B_0^*$. Therefore, an element $z\in A(R)^*$ can be written as $z=z_1+z_2$, $z_1\in A(R)^*_f$, $z_2\in A(R)_0^*$. Now, define 
$$ \theta: A(R)^* \to B^*, \quad \theta(z_1+z_2)= \phi_e(z_1)+\psi_d(z_2)$$
We need to check that $\theta$ is well-defined, which in this case requires us to check whether $\phi_e$ and $\psi_d$ agree on $B_f^*\cap B_0^*$. Elements in the intersection, if they exist, are generated by $e_ia_i=b_i$ in the case they belong to $A(R)_0^*$. If $e_ia_i\in A(R)^*_0$, then  $$\phi_e(e_ia_i)= e_iu_i=d_iv_i=\psi_d(b_i).$$  
The image of $\theta$ has index $d=d_1\cdots d_n$ in $B^*$. In terms of the elements $x_1,\ldots, x_m$ and $y_1,\ldots, y_m$ described earlier, 
$\theta(x_j)=y_j$, unless $j=m_3+i$, $i=1, \ldots , n$, in which case $x_j=b_i$, $y_j=v_i$, and $\theta(b_i)=d_iv_i$.
    \end{proof}
    The following lemma is a classical result implied by the structure theory of saturated abelian groups. See~\cite{eklof} for a reference. 
    \begin{lemma}\label{satabelian:lem} Assume $A\cong A_1 \oplus \cdots \oplus A_n$, where each $A_i$ is a cyclic module over $R$ and $R$ is a saturated model of $\Z$. Let $v_1, \ldots ,v_n$ be the corresponding basis. Let $\alpha_i$, $i=1, \ldots, n$ be elements of $R$ such that $p\nmid \alpha_i$ for any prime number $p$, where $|$ denotes the division in the ring $R$. Then, there exists an automorphism of $A$ as an abelian group that extends to $v_i\mapsto \alpha_iv_i$. \end{lemma} 
\begin{proof} Since $R^+$ is a saturated abelian group, then $$\displaystyle  R^+\cong \Q^\beta \oplus \sum_{p\in \pi} (\Z_p)^{\alpha_p},$$where $\Q$ is the additive group of rationals, $\Z_p$ is the additive group of $p$-adic integers, $\pi$ is the set of all primes, and the cardinals $\beta$ and $\alpha_p$ depend on $R^+$. Clearly, if the lemma condition applies to $\alpha$, then $1\mapsto \alpha$ can be extended to an automorphism of every summand. This is not necessarily an $R$-module automorphism.\end{proof} 

\begin{theorem} Assume $B=(A(R),B_0,f,g,h)$, where $B_0\cong A(R)_0$. Then for any $\aleph_1$-incomplete ultrafilter $(\N,D)$, $B^*\cong A(R)^*$.    
\end{theorem}

\begin{proof}
We shall use the monomorphism $\theta: A(R)\to B$ constructed in Lemma~\ref{lem:mono} to construct the claimed isomorphism. We will use the terminology and notation of Lemma~\ref{lem:mono}.  Recall that $e=e_{1}\cdots e_{n}$, $d=d_1\cdots d_n$, and gcd$(d,e)=1$. Assume $\pi$ denotes the set of all prime numbers. For each $i=1, \ldots, n$, let $\pi_i$ denote the set of all prime numbers $p$, such that $p|d_i$. Let us denote the $j$'th prime number in $\pi\setminus \pi_i$ by $p_{ij}$ and the product of the first $j$ primes in $\pi\setminus \pi_i$ by $q_{ij}$.
 
       For each $j\in \N$ and each $1\leq i \leq n$, define 
 $$w_{ij}=(d_i+q_{ij}e) v_i$$
 Now, let $w_i^*$ and $q_i^*$ denote the classes of $(w_{ij})_{j\in \N}$ and $(q_{ij})_{j\in \N}$, respectively, in $\Z^*$. Indeed $$w^*_i=(d_i+q_i^*e)v_i.$$ Set $\alpha_i=d_i+q_i^*e$ for each relevant $i$.

 Next, we claim that $\alpha_i$ satisfies the hypothesis of Lemma~\ref{satabelian:lem}; that is, no prime $p$ divides $\alpha_i=d_i+q_i^*e$ for each $i$. Pick a prime $p$. If $p\in \pi_d$, i.e., $p|d_i$ and $p|(d_i+q_{ij}e)$, then $p|q_{ij}e$, which contradicts the choice of $q_{ij}$ and the fact that gcd$(d_k,e)=1$. So for such $p$, $p\nmid (d_i+q_{ij}e)$. Now choose a prime $p\in \pi\setminus \pi_i$, i.e, $p\nmid d_{i}$. Then $p=p_{it}$ for some $t\in \N$, implying that $p$ is a factor of $q_{ij}$ for every $j\geq t$. So $p|q_{ij}e$ for every $j\geq t$. Therefore, for every such $j$, if $p|(d_i+q_{ij}e)$ then $p|d_i$, which is impossible. So for every $j\geq t$, $p\nmid d_i+q_{ij}e$. Therefore, for any prime $p$, $p\nmid (d_i+q_{i}^*e)$.
 Now we twist $\theta$ defined above into maps $\theta_j$ in such a way that the cocycle $g$ (or $g^*$) does not change modulo coboundaries; therefore, each $\theta_j$
 defines a monomorphism of $A(R)$ into $B$, but with different indices.
  
       Let $\hat{e}_i=e/e_i$. For each $j\in\N$ and each $k=1, \ldots ,m$, let
\begin{equation*}
  \theta_j(x_k)=\left\{
  \begin{array}{ll}
   y_k & \text{if } x_k\neq a_i, b_i; i=1, \ldots ,n \\
  u_i+ q_{ij}\hat{e}_iv_i& \text{if } x_k=a_i; i=1,\ldots, n\\
  d_iv_i+ q_{ij}ev_i & \text{if } x_k=b_i; i=1,\ldots, n\end{array}\right.
   \end{equation*}
 The above assignment can clearly be extended to a homomorphism from $A(R)^*$ into $B^*$. Note that here $\theta_j$ picks a transversal $u_i+B_0$ which produces a cocycle cohomologous to $g$ defined by the transversal $e_iu_i= d_iv_i$ since $d_iv_i \equiv d_iv_i+ q_{ij}ev_i$ (mod $e$) and therefore by Lemma~\ref{cocycle:lem} they define cocycles which are in the same class in $\Ext(\Z/e_i\Z, \Z^n)$. Let $\theta^*:(A(R))^*\to B^*$ be the monomorphism induced $(\theta_j)_{j\in \N}$. Now $\theta^*$ is an injective homomorphism that is also surjective by Lemma~\ref{satabelian:lem}. This finishes the proof.
\end{proof}
The above result concludes the proof of the converse part of the Characterization Theorem~\ref{thm:char}.

\section*{Acknowledgments}
This paper appears in a special volume in honor of Alexei G. Myasnikov (Miasnikov), and it is a pleasure to acknowledge how deeply his influence has shaped my mathematical life. I first met Alexei in 2003, in his McGill office, where he introduced me to the Mal'cev correspondence between rings and unitriangular groups. I was immediately captivated by the elegance of the constructions and by Alexei's gift for making ideas feel both simple and profound.

Over the years, Alexei has been my supervisor, collaborator, and friend. His originality, generosity, and unwaveringly positive attitude toward scientific and personal challenges have been a constant source of inspiration. The results of this paper are offered in that spirit as a contribution toward the broader program that Alexei helped set in motion.

\end{document}